\documentclass[11pt,leqno,tbtags]{amsart}
\usepackage{amssymb}
\usepackage{natbib}
\bibpunct[, ]{[}{]}{;}{n}{,}{,}
\usepackage{url}

\numberwithin{equation}{section}
\allowdisplaybreaks

\newcommand\urladdrx[1]{{\urladdr{\def~{{\tiny$\sim$}}#1}}}

\title{Graph limits and hereditary properties}
\date{February 17, 2011; revised and extended March 28, 2013}

\author{Svante Janson}
\address{Department of Mathematics, Uppsala University, PO Box 480,
SE-751~06 Uppsala, Sweden}
\email{svante.janson@math.uu.se}
\urladdrx{http://www.math.uu.se/~svante/}
\thanks{Research partly supported by the Knut and Alice Wallenberg Foundation.}
\thanks{A preliminary version was presented at the workshop
\emph{Graph limits, homomorphisms and structures II} at
Hrani{\v c}n{\'i} Z{\'a}me{\v c}ek,  Czech Republic, 2012.
I thank the participants, and Anders Johansson, for helpful comments.
}


\newtheorem{theorem}{Theorem}[section]
\newtheorem{lemma}[theorem]{Lemma}

\newtheorem{corollary}[theorem]{Corollary}

\theoremstyle{definition}
\newtheorem{example}[theorem]{Example}

\newtheorem{problem}[theorem]{Problem}

\newtheorem{remark}[theorem]{Remark}
\newtheorem*{remark*}{Remark}

\theoremstyle{remark}

\newenvironment{romenumerate}{\begin{enumerate}
 }{\end{enumerate}}

\newcounter{oldenumi}
{\setcounter{oldenumi}{\value{enumi}}
\begin{romenumerate} \setcounter{enumi}{\value{oldenumi}}}
{\end{romenumerate}}

\newcounter{prop}
\newenvironment{property}%
{\setlength{\leftmargini}{30pt}%
\begin{enumerate}%
\setlength{\labelsep}{\leftmargini}
\stepcounter{prop}%
\item\em}
{\end{enumerate}}

\newcounter{thmenumerate}
\newenvironment{thmenumerate}
{\setcounter{thmenumerate}{0}%
 \def\item{\par
 \refstepcounter{thmenumerate}\textup{(\roman{thmenumerate})\enspace}}
}
{}

\newcounter{xenumerate}   

\newcommand\pfitemref[1]{\par\ref{#1}:}

\newcommand{\refT}[1]{Theorem~\ref{#1}}
\newcommand{\refC}[1]{Corollary~\ref{#1}}
\newcommand{\refL}[1]{Lemma~\ref{#1}}
\newcommand{\refR}[1]{Remark~\ref{#1}}
\newcommand{\refS}[1]{Section~\ref{#1}}

\newcommand{\refE}[1]{Example~\ref{#1}}

\begingroup
  \divide\count255 by 60
  \multiply\count255 by -60
  \advance\count255 by \time
  \ifnum \count255 < 10 \xdef\klockan{\the\count1.0\the\count255}
  \else\xdef\klockan{\the\count1.\the\count255}\fi
\endgroup

\newcommand\nopf{\qed}   

\newcommand{\sumi}{\sum_{i=1}^\infty}

\newcommand\set[1]{\ensuremath{\{#1\}}}

\newcommand\bigpar[1]{\bigl(#1\bigr)}

\newcommand\lrpar[1]{\left(#1\right)}

\def\rompar(#1){\textup(#1\textup)}    

\def\xexp(#1){e^{#1}}
\newcommand\ceil[1]{\lceil#1\rceil}
\newcommand\floor[1]{\lfloor#1\rfloor}

\newcommand\setn{\set{1,\dots,n}}
\newcommand\ntoo{\ensuremath{{n\to\infty}}}

\newcommand\asntoo{\text{as }\ntoo}

\newcommand\iid{i.i.d.\spacefactor=1000}    
\newcommand\ie{i.e.\spacefactor=1000}
\newcommand\eg{e.g.\spacefactor=1000}
\newcommand\viz{viz.\spacefactor=1000}
\newcommand\cf{cf.\spacefactor=1000}
\newcommand{\as}{a.s.\spacefactor=1000}
\newcommand{\aex}{a.e.\spacefactor=1000}

\newcommand{\tend}{\longrightarrow}
\newcommand\dto{\overset{\mathrm{d}}{\tend}}
\newcommand\pto{\overset{\mathrm{p}}{\tend}}

\newcommand\bbR{\mathbb R}

\newcommand\bbN{\mathbb N}

\newcounter{CC} 
\newcounter{cc}

\newcommand\E{\operatorname{\mathbb E{}}}
\renewcommand\P{\operatorname{\mathbb P{}}}

\newcommand\ga{\alpha}

\newcommand\gd{\delta}
\newcommand\gD{\Delta}
\newcommand\gf{\varphi}
\newcommand\gam{\gamma}
\newcommand\gG{\Gamma}

\newcommand\gl{\lambda}

\newcommand\gs{\sigma}

\newcommand\cA{\mathcal A}

\newcommand\cC{\mathcal C}
\newcommand\cD{\mathcal D}

\newcommand\cF{\mathcal F}
\newcommand\cI{\mathcal I}
\newcommand\cJ{\mathcal J}

\newcommand\cP{\mathcal P}
\newcommand\cQ{\mathcal Q}
\newcommand\cS{{\mathcal S}}
\newcommand\cT{{\mathcal T}}
\newcommand\cU{{\mathcal U}}

\newcommand\ett[1]{\boldsymbol1\set{#1}} 
\newcommand\etta{\boldsymbol1} 
\def\[#1]{[\![#1]\!]}

\newcommand\qqq{^{1/3}}
\newcommand\qqqw{^{-1/3}}
\newcommand\qqqb{^{2/3}}

\renewcommand{\=}{:=}

\newcommand\oi{[0,1]}
\newcommand\setoi{\set{0,1}}

\newcommand\dtv{d_{\mathrm{TV}}}

\newcommand\dd{\,\textup{d}}

\newcommand{\Lovasz}{Lov\'asz}

\newcommand{\tind}{t_{\mathrm{ind}}}

\newcommand{\cuq}{\overline{\cU}}
\newcommand{\cuqf}{\overline{\cU_\cF}}
\newcommand{\bcu}{\overline{\cU}}
\newcommand{\bbcu}{\widehat{\cU}}

\newcommand{\xoo}[1]{\widehat{#1}}
\newcommand{\cuoo}{\xoo{\cU}}

\newcommand{\cpoo}{\xoo{\cP}}
\newcommand{\cqoo}{\xoo{\cQ}}

\newcommand{\bcp}{\overline{\cP}}
\newcommand{\bct}{\overline{\cT}}

\newcommand{\dcut}{\gd_\square}

\newcommand{\ps}{probability space}

\newcommand{\sss}{\cS}
\newcommand{\sssq}{{\cS^2}}

\newcommand{\aig}{$\cA$-inter\-sec\-tion graph}
\newcommand{\aigs}{$\cA$-inter\-sec\-tion graphs}
\newcommand{\aigl}{$\cA$-intersection graph limit}

\newcommand\oivalued{\setoi-valued} 
\newcommand{\UI}{\mathcal{UI}}
\newcommand{\CR}{\mathcal{CR}}
\newcommand{\cfx}{{\cF^*}}
\newcommand{\cUU}{\cU^*}
\newcommand{\pa}{\cP_{\cA}}
\newcommand{\bpa}{\overline{\cP_{\cA}}}
\newcommand{\dc}{disjoint clique}
\newcommand{\dcg}{disjoint clique graph}
\newcommand{\dcgl}{disjoint clique graph limit}
\newcommand{\DC}{\mathcal{DC}}
\newcommand{\DCoo}{\xoo{\DC}}
\newcommand{\dci}{\mathcal{DC}_1}
\newcommand{\iii}{_{i=1}^\infty}

\newcommand{\bcn}{\textbf{c}_n}
\newcommand{\bs}{\textbf{s}}

\newcommand{\MP}{\mathcal{M}}
\newcommand{\LG}{\mathcal{LG}}
\newcommand{\nn}{^{(n)}}
\newcommand{\lgl}{line graph limit}
\newcommand{\cflg}{\cF_L}
\newcommand\xxm{x_1,\dots,x_m}

\newcommand{\SP}{\mathcal{SP}}
\newcommand{\hSP}{\widehat{\SP}}
\newcommand{\ggw}{[W]}
\newcommand{\ggx}[1]{[#1]}
\newcommand{\ggo}{\ggx0}
\newcommand{\ggi}{\ggx1}
\newcommand{\ggd}{\ggMP}
\newcommand{\CH}{\mathcal{CH}}
\newcommand{\wch}{W^{\CH}}

\newcommand{\wMP}{W^{\MP}}
\newcommand{\ws}{\wMP_{\bs}}
\newcommand{\ggMP}{\gG^{\MP}}
\newcommand{\ggs}{\ggMP_{\bs}}

\newcommand\REM[1]{{\raggedright\texttt{[#1]}\par\marginal{XXX}}}

\begin{document}

\begin{abstract} 
We collect some general results on graph limits associated to hereditary
classes of graphs.
As examples, we consider some classes 
defined by forbidden subgraphs and some classes
of intersection graphs, including 
triangle-free graphs,
chordal graphs,
cographs,
interval
graphs, unit interval graphs, threshold graphs, 
and line graphs.
\end{abstract}

\maketitle

\section{Introduction}\label{S:intro}

We use standard concepts from the theory of graph limits, see
\eg{} \cite{LSz,BCLSVi,BCLSVii,SJ209} and the recent book by \citet{Lovasz}.
(We use here mainly the notation of \cite{SJ209}.)
$\cU$ is the set of all unlabelled graphs
(all our graphs are finite and simple), and this is embedded (as a
countable, discrete, dense, open) subset of a compact metric space $\bcu$;
the complement $\cuoo\=\bcu\setminus\cU$ is the set of graph limits, 
which thus itself is a compact metric space.
If $F$ and $G$ are graphs, then the \emph{homomorphism} or \emph{subgraph
  number} $t(F,G)$  
is the probability that a uniformly random mapping $\gf$ of the vertex set
$V(F)$ into 
$V(G)$ is a graph homomorphism, \ie, $u \sim v \implies \gf(u)\sim\gf(v)$
for $u,v\in V(F)$;
similarly, the \emph{induced subgraph number} $\tind(F,G)$ is 
(when $|F|\le |G|$) 
the probability that
a uniformly random injective mapping $V(F)\to V(G)$ is a graph isomorphism
onto an induced subgraph, \ie, $u \sim v \iff \gf(u)\sim\gf(v)$,
The subgraph numbers $t(F,G)$ 
and the induced subgraph numbers $\tind(F,G)$
extend from graphs $G\in\cU$ to general $G\in\cuq$ by continuity
(see the
references above for details, and note that $F$ always is a graph, which
we regard as fixed).
Moreover, the topology of $\cuq$ can be described by 
these numbers $t(F,G)$ 
or $\tind(F,G)$, and
a sequence of graphs $(G_n)$ converges to a graph limit $\gG$
$\iff$
$|G_n|\to\infty$ and  
$t(F,G_n)\to t(F,\gG)$ for every graph $F$
$\iff$
$|G_n|\to\infty$ and  
$\tind(F,G_n)\to \tind(F,\gG)$ for every graph $F$.
Furthermore, a graph limit $\gG$ is uniquely determined by the numbers
$t(F,\gG)$
(or $\tind(F,\gG)$) for  $F\in\cU$.

A \emph{graph class} is a subset of the set $\cU$ of unlabelled graphs,
\ie, a class of graphs closed under isomorphisms.
Similarly, a \emph{graph property} is a property of graphs that does not
distinguish between isomorphic graphs; there is an obvious 1--1
correspondence between graph classes and graph properties and we will not
distinguish between a graph property and the corresponding class. 
A graph class or property $\cP$ is \emph{hereditary} if whenever a graph $G$ has
the property $\cP$, then every induced subgraph of $G$ also has $\cP$;
this can be written
\begin{equation}\label{her}
 G \in\cP \text{ and } \tind(F,G)>0 \implies F\in\cP.
\end{equation}

Many examples of hereditary graph classes are given in \eg{}
\cite{Brand} and \cite{Golumbic}.

\begin{example}[intersection graphs]\label{Einter}
Many interesting hereditary graph classes are given by 
various classes of intersection graphs.
In general, 
we consider a collection $\cA$ of subsets of some
universe and say that a graph $G$ 
is an \aig{} if 
there exists a collection of sets $\set{A_i}_{i\in V(G)}\in\cA$ 
such that
there is an edge $ij\in E(G)$ if and only if $A_i\cap A_j\neq0$.
The class $\pa$ of all \aig{s} is a hereditary graph class, for any $\cA$.

Specific
examples are the classes of interval graphs, 
unit interval graphs,
circular-arc graphs, circle graphs and permutation graphs 
studied in
\cite{SJ254}. See \eg{} \cite{Brand} and
\cite{Golumbic} for several further examples.
\end{example}

\begin{example}[forbidden subgraphs]\label{Eforbidden}
  If $\cF$ is a (finite or infinite) family of (unlabelled) graphs,
 let $\cU_\cF$ be the class of all graphs that 
do not contain any graph from $\cF$ as an induced subgraph:
Then  $\cU_\cF$ is a hereditary graph class. We will study this type of
graph classes in \refS{Sforbidden}, where also several examples are given.
\end{example}

\begin{example}[monotone properties]\label{Emon}
  A property $\cP$ is \emph{monotone} if whenever a graph $G$ has the
  property $\cP$,
  then so does every (not necessarily induced) subgraph of $G$. In other
  words, \cf{} \eqref{her}, 
\begin{equation}\label{mon}
 G \in\cP \text{ and } t(F,G)>0 \implies F\in\cP.
\end{equation}
Obviously, every monotone property is hereditary. 
\end{example}

Let $\cP\subseteq\cU$ be a graph class. We let $\bcp\subseteq \bcu$ be the
closure of $\cP$ in $\bcu$, and $\cpoo\=\bcp\cap\cuoo$ the set of graph
limits of graphs in $\cP$. Explicitly, $\cpoo$ is the set of graph limits $\gG$
such that there exists a sequence of graphs $G_n$ in $\cP$ with $G_n\to\gG$.
(Note that we also use $\overline{G}$ to denote the complement of a graph
$G$; this should not cause any confusion.)

\begin{remark}
Since $\cU$ is open and discrete in $\cuq$, \ie{} every element of $\cU$ is
isolated in $\bcu$, we trivially have
$\bcp\cap\cU=\cP$; thus $\bcp=\cP\cup\cpoo$. 
If $\gG$ is a graph limit, 
then $\gG\in\cpoo$ and $\gG\in\bcp$ are equivalent, and we will use both
formulations interchangeably.
\end{remark}

It seems to be of interest to study the classes $\cpoo$ of graph limits
defined by various graph properties.
General results are given in \citet{LSz:regularity} and \citet{HJSz}.
Some examples have been studied, on a case-by-case basis, 
in \cite{SJ238}
(threshold graphs) and \cite{SJ254} (interval graphs and some related
graph classes), and there are many other classes that could be studied;
apart from the intrinsic interest of various graph classes,
some further general patterns might emerge from the study of individual classes.

The purpose of this note is to collect a few general remarks, results and
examples; 
some of them from the literature and some of them new.
Some further notions and facts from graph limit theory are recalled in
\refS{Sgraphons}. 
In \refS{Srandom} we characterize
graph limits of hereditary classes of graphs 
using random graphs.
\refS{Sforbidden} studies the case of classes defined by forbidding certain
subgraphs, and 
\refS{Sinter} studies classes of intersection graphs.
\refS{Srf} treats the notion of random-free graph classes,
introduced by \citet{LSz:regularity}.
Sections \ref{Sdc}--\ref{Sclaw} consider some further (rather simple) examples,
including (\refS{SRamsey})  graph limit versions of Ramsey's theorem.

We note a simple but useful fact about a trivial case.

\begin{theorem}
  \label{T0}
Let $\cP$ be an arbitrary graph class. Then the following are equivalent.
\begin{romenumerate}
\item 
$\cP$ is finite.
\item 
There exists $n_0$ such that $|G|\ge n_0\implies G\notin\cP$.
\item 
$\cpoo=\emptyset$.
\end{romenumerate}
\end{theorem}

\begin{proof}
  (i)$\iff$(ii) is obvious, since the set \set{G\in\cU:|G|=n} is finite for
  every $n$.

(i)$\implies$(iii):
If $\cP$ is finite, then $\cP$ is a closed set in $\bcu$,
so $\bcp=\cP\subset\cU$ and thus
$\cpoo=\bcp\setminus\cU=\emptyset$.

(iii)$\implies$(ii):
If (ii) does not hold, then there is a sequence $G_n\in\cP$ with
$|G_n|\to\infty$. Then some subsequence converges, and its limit is an
element of $\cpoo$, so $\cpoo\neq\emptyset$.
\end{proof}

\begin{remark}
We have so far discussed the set $\cpoo$
of all possible limits of sequences of
graphs in a class $\cP$. Another interesting problem is to take
a uniformly random  graph $G_n$ in $\cP_n\=\set{G\in\cP:|G|=n}$ and study
its asymptotic behaviour. 
Is there a random graph limit $\gG$ such that $G_n\dto \gG$, where $\dto$
denotes convergence in distribution (as random elements of the metric space
$\bcp$)?
In particular, does there exist a single graph limit
$\gG$ (necessarily in $\cpoo$) such that $G_n\to\gG$ in probability?
There are actually two versions of this problem, since one may take $G_n$
either labelled or unlabelled; the classes of graphs is the same but the
distribution of a uniform unlabelled graph (with $n$ vertices) in $\cP$
differs in general from the
distribution of a uniform labelled graph (with vertex set $[n]$) in $\cP$,
and it is possible that the limits might differ. (An example is given in
\refE{EDC1}. Nevertheless we usually expect the same limit, since most
graphs have a trivial automorphism group.) 
We give some remarks on this problem in a few examples.
The problem of limits of random graphs is studied further in \cite{HJSz},
where it is connected to the entropy of graph limits.  
\end{remark}

\section{Graphons and random graphs}\label{Sgraphons}

A graph limit $\gG$ can be represented by a \emph{graphon}, which is a
symmetric measurable function $W:\sssq\to\oi$ for some probability space
$(\sss,\mu)$. 
(Often, but not always, taken as $(\oi,\gl)$, where $\gl$ is Lebesgue measure.)
Note that the representation is far from unique, see \eg{}
\cite{BCL:unique} and \cite{SJ249}. 
One defines, for a graph $F$ and a graphon $W$,
\begin{align}
  t(F,W)
&\=
\int_{\cS^{|F|}} \prod_{ij\in E(F)} W(x_i,x_j) 
 \dd\mu(x_1)\dotsm \dd\mu(x_{|F|}),
\label{tfw}
\\
\tind(F,W)
&\=
\int_{\cS^{|F|}} \prod_{ij\in E(F)} W(x_i,x_j) \prod_{ij\notin E(F)} (1-W(x_i,x_j))
 \dd\mu(x_1)\dotsm \dd\mu(x_{|F|}),
\label{tindfw}
\end{align}
and 
the graphon $W$ represents the graph limit $\gG$ that has $t(F,\gG)=t(F,W)$
for every graph $F$ (or, equivalently, $\tind(F,\gG)=\tind(F,W)$ for every $F$).

Unlike graph limits, graphons are not uniquely determined by the
homomorphism
numbers $t(F,W)$; we say that two graphons $W$ and $W'$ (possibly defined on
different \ps{s}) are \emph{equivalent} if they represent the same graph
limit, \ie, if $t(F,W)=t(F,W')$ for all graphs $F$.
(For other characterizations of  equivalent graphons, see \eg{}
\citet{BCL:unique}, \citet{BR}, \citet{SJ249}.)
If $F$ has connected components $F_1,\dots,F_m$, then
\begin{equation}
  \label{fcomp}
t(F,W)=\prod_{i=1}^m t(F_i,W);
\end{equation}
hence it suffices here (and for many other purposes)
to consider connected $F$.

Graph limits may thus be regarded as equivalence classes of graphons. We
write (following \cite{Pikhurko})
$\ggw$ for the graph limit represented by a graphon $W$. It is often
convenient to represent graph limits by graphons, and we may, for example, 
write  $G_n\to W$ for $G_n\to\ggw$.

Let $X_1,X_2,\dots$ be an \iid{} sequence of random elements of $\sss$ with
distribution $\mu$. Then \eqref{tfw}--\eqref{tindfw} can be written more
concisely as 
\begin{align}\label{tfwp}
  t(F,W)
&=
\E \prod_{ij\in E(F)} W(X_i,X_j) ,
\\
\tind(F,W)
&=
\E\lrpar{ \prod_{ij\in E(F)} W(X_i,X_j) \prod_{ij\notin E(F)} (1-W(X_i,X_j))}.
\label{tindfwp}
\end{align}

A graphon defines a random graph $G(n,W)$ 
with vertex set $[n]\=\setn$
for every $n\ge1$ by a standard construction:
let $X_1,X_2,\dots$ be as above, and given $X_1,\dots,X_n$, let $ij$ be
an edge with probability $W(X_i,X_j)$, independently for all pairs $(i,j)$
with $1\le i< j\le n$. It follows by \eqref{tindfwp} that
if $F$ is any graph with vertex set $[n]$, then 
\begin{align}
  \P\bigpar{G(n,W)=F}&=\tind(F,W); \label{pgnwind}
\intertext{equivalently, see \eqref{tfwp},}
  \P\bigpar{G(n,W)\supseteq F}&=t(F,W). \label{pgnw}
\end{align}
This shows that the random graph $G(n,W)$ is the same 
(in the sense that the distribution is the same) 
for all graphons
representing the same graph limit $\gG$.
Thus every graph limit $\gG$ defines a random graph $G(n,\gG)$ with vertex
set $[n]$ for every $n\ge1$, and this random graph can 
by \eqref{pgnwind}
be defined directly
by the formula
\begin{equation} \label{sofie}
  \P\bigpar{G(n,\gG)=F}=\tind(F,\gG)
\end{equation}
for every graph $F$ on $[n]$, which gives the distribution.

It is shown by \citet[Theorem  4.5]{BCLSVi} that
as \ntoo, the random graph $G(n,W)$ converges a.s.\ to $W$. Thus,
\begin{equation}
  \label{gnwlim}
G(n,\gG)\pto\gG,
\qquad \asntoo.
\end{equation}

\section{Graph limits and random graphs}\label{Srandom}

Graph limits in $\bcp$ (or equivalently, in $\cpoo\=\bcp\cap\cuoo$)
can be characterized by the random graphs $G(n,\gG)$.

\begin{theorem}\label{Tgngg}
  Let $\cP$ be a hereditary graph class and let $\gG$ be a graph limit.
Then $\gG\in\bcp$ if and only if
$G(n,\gG)\in\cP$ a.s.\ for every $n\ge1$.
\end{theorem}

This is an immediate consequence of the following more detailed result.

\begin{theorem}\label{Tgngg2}
  Let $\cP$ be a hereditary graph class and let $\gG$ be a graph limit.
Then one of the following alternatives hold:
\begin{romenumerate}
\item 
 $\gG\in\bcp$ and 
$G(n,\gG)\in\cP$ a.s.\ for every $n\ge1$.
\item 
 $\gG\notin\bcp$ and 
$\P\bigpar{G(n,\gG)\in\cP}\to0$ as \ntoo.
\end{romenumerate}
\end{theorem}

\begin{proof}
Let $\gG\in\bcp$ and 
  suppose that $G_n\to\gG$ with $G_n\in\cP$. If $F\notin\cP$, then
  $\tind(F,G_n)=0$ for every $n$ by \eqref{her}, and thus, by \eqref{sofie},
$$\P(G(n,\gG)=F)=\tind(F,\gG)=\lim_\ntoo \tind(F,G_n)=0.$$ 

Conversely, if $\gG\notin\bcp$, then there is an open neighbourhood $V$ of
$\gG$ in $\bcu$ such that $V\cap\cP=\emptyset$. As said above,
$G(n,\gG)\to\gG$ a.s., which implies convergence in probability.
Thus $\P\bigpar{G(n,\gG)\in V}\to1$ and 
$\P\bigpar{G(n,\gG)\in\cP}
\le\P\bigpar{G(n,\gG)\notin V}
\to0$ as \ntoo.
\end{proof}

We obtain a couple of easy corollaries of \refT{Tgngg}.

\begin{theorem}\label{Tsofie}
  Let $\cP$ be a hereditary graph class and let $\gG$ be a graph limit.
Then the following are equivalent:
\begin{romenumerate}
\item \label{tsofie1}
$\gG\in\bcp$. 
\item \label{tsofie>0}
$\tind(F,\gG)>0\implies F\in\cP$.
\item \label{tsofie=0}
If $F\notin\cP$, then $\tind(F,\gG)=0$.
\end{romenumerate}
\end{theorem}

\begin{proof}
  Immediate by \refT{Tgngg} and \eqref{sofie}.
\end{proof}

\begin{theorem}\label{Tcap}
  Let \set{\cP_\ga} be a finite or infinite family of hereditary graph
  classes and let $\cP=\bigcap_{\ga}\cP_\ga$.
Then $\bcp=\bigcap_{\ga}\bcp_\ga$.
\end{theorem}

Thus, for example, if a graph limit is the limit of some sequence $G_n$ of
graphs in $\cP_1$, and also of another such sequence $G_n'$ in $\cP_2$, then
it is the limit of some sequence $G_n''$ in $\cP_1\cap\cP_2$.
(This is not true in general, without the assumption that the classes are
hereditary. For example, let $\cP_2=\cU\setminus\cP_1$, with, say,
$\cP_1$ the class of interval graphs.)

\begin{proof}
  Suppose that $\gG\in \bigcap_{\ga}\bcp_\ga$. 
If $\tind(F,\gG)>0$, then \refT{Tsofie}\ref{tsofie1}$\Rightarrow$\ref{tsofie>0} 
shows that $F\in \cP_\ga$ for every $\ga$;
hence $F\in \cP$.
Consequently, \refT{Tsofie} in the opposite direction shows that $\gG\in\bcp$. 
The converse is obvious.
(Alternatively, one can use \refT{Tgngg} directly, with a little care if the
family \set{\cP_\ga} is uncountable.)
\end{proof}

\begin{remark}
  Conversely, we may ask whether every graph in $\cP$ can be obtained (with
  positive probability) as $G(n,\gG)$ for some $\gG\in\bcp$ and some $n$.
By \eqref{sofie}, the class of graphs obtainable in this way equals
$\bigcup_{\gG\in\bcp}\cI(\gG)$ where
\begin{equation}
  \cI(\gG)\=\set{F\in\cU:\tind(F,\gG)>0}.
\end{equation}
By \refT{Tsofie} (see also \cite{LSz:regularity}), 
$\bigcup_{\gG\in\bcp}\cI(\gG)\subseteq\cP$ for every hereditary graph class
$\cP$. 
\citet{LSz:regularity} have shown that equality holds if and only if $\cP$ has
the following property:
\begin{property}\label{Ptwin}
If $G\in\cP$ and $v$ is a vertex in $G$, and 
we enlarge $G$ 
by adding a twin $v'$ to $v$, \ie, a new vertex with the same
neighbours as $v$, 
then at least one of the two graphs
obtained by further either adding or not adding an edge $vv'$
belongs to $\cP$.
\end{property}
\end{remark}

\section{Forbidden subgraphs}\label{Sforbidden}

If $\cF$ is a (finite or infinite) family of (unlabelled) graphs,
we let $\cU_\cF$ be the class of all graphs that 
do not contain any graph from $\cF$ as an induced subgraph,
\ie,
\begin{equation}\label{uf}
\cU_\cF\=\set{G\in\cU:\tind(F,G)=0\text{ for } F\in\cF}.  
\end{equation}
This is evidently a hereditary class.

We similarly define
\begin{equation}\label{quf}
\cuq_\cF\=\set{\gG\in\cuq:\tind(F,\gG)=0 \text{ for } F\in\cF},
\end{equation}
and have the following simple result (\cite[Theorem 3.2]{SJ238}).

\begin{theorem}\label{Tquf}
Let $\cU_\cF$ be given by \eqref{uf}.
Then
$\cuqf=\cuq_\cF$.
In other words, if $\gG\in\cuoo$ is a graph limit, then $\gG$ is a limit of
a sequence of graphs in $\cU_\cF$ if and only if 
$\tind(F,\gG)=0$ for $F\in\cF$.
\end{theorem}

\begin{proof}
 If $G_n\to\gG$ with $G_n\in\cU_\cF$, then
  $\tind(F,\gG)=\lim_\ntoo \tind(F,G_n)=0$ for every
  $F\in\cF$, by \eqref{uf} and the continuity of $\tind(F,\cdot)$.
Thus $\gG\in\cuq_\cF$.

Conversely, suppose that $\gG\in\cuoo$ and
$\tind(F,\gG)=0$ for $F\in\cF$. 
It follows from \eqref{sofie} that if
$F\in\cF$ then, for any $n\ge1$, $G(n,\gG)\neq F$ a.s.;
thus $G(n,\gG)\notin\cF$ a.s.
Moreover, every induced subgraph of $G(n,\gG)$ has the same distribution as
$G(m,\gG)$ for some $m\le n$; hence a.s.\ no induced subgraph belongs to
$\cF$ and 
thus $G(n,\gG)\in\cU_\cF$.
Hence $\gG\in\cuqf$ by \refT{Tgngg}.
\end{proof}

\begin{remark}
Every hereditary class of graphs $\cP$ is of the form $\cU_\cF$ for some $\cF$;
  we can simply take $\cF\=\cU\setminus\cP$.
(In this case, \refT{Tquf} reduces to \refT{Tsofie}.)
However, we are mainly interested in cases when $\cF$ is a small family.
\end{remark}

\begin{example}\label{EK3}
The class of \emph{triangle-free graphs} is
 $\cU_{\set{K_3}}$.
\refT{Tquf} shows that the triangle free
  graph limits
(\ie, the limits of triangle free graphs) 
are the graph limits $\gG$ with 
  $\tind(K_3,\gG)=0$.
\end{example}

\begin{example}\label{Echordal0}
    The class of \emph{chordal graphs} or \emph{triangulated graphs}
is the class of all graphs not containing any induced $C_k$ with $k\ge4$,
\ie, $\cU_{\set{C_4,C_5,\dots}}$. 
(See \refE{Echordal} below and \cite[Section 1.2]{Brand} and
\cite[Chapter 4]{Golumbic} for other equivalent characterizations (and further
names).)
By \refT{Tquf}, a graph limit $\gG$ is a chordal graph limit if and only if
$\tind(C_k,\gG)=0 $ for every $k\ge4$.
\end{example}

\begin{example}\label{ECR}
  The class $\CR$ of \emph{cographs} equals $\cU_{\set{P_4}}$, see 
\cite[in particular Theorem 11.3.3]{Brand} where several alternative
characterizations are given.
Thus, if $\gG$ is a graph limit,
$\gG\in\overline{\CR}$  if and only if
$\tind(P_4,\gG)=0$. Such graph limits are studied in 
\citet{LSz:finitely}.
\end{example}

\begin{example}[\citet{SJ238}]
  The class $\cT$ of \emph{threshold graphs} equals $\cU_{\set{2K_2,P_4,C_4}}$
Thus, if $\gG$ is a graph limit, then
$\gG\in\bct$ if and only if
$\tind(P_4,\gG)=\tind(C_4,\gG)=\tind(2K_2,\gG)=0$.
\end{example}

\begin{example}
The class $\UI$ of \emph{unit interval graphs}
equals 
the class of graphs that contain no induced subgraph isomorphic to
$C_k$ for any  $k\ge4$, $K_{1,3}$, $S_3$ or $\overline S_3$, where $S_3$ is
the graph on 6 vertices \set{1,\dots,6} with edge set
$\set{12,13,23,14,25,36}$, and $\overline S_3$ is its complement
\cite[Theorem 7.1.9]{Brand}. 
Thus, if $\gG$ is a graph limit,
$\gG\in\overline{\UI}$  if and only if
$\tind(F,\gG)=0$ for every  
$F\in\set{C_k}_{k\ge4}\cup\set{K_{1,3},S_3,\overline S_3}$.
\end{example}

Further examples are studied in Sections \ref{Sdc}--\ref{Sclaw}.

We obtain just as easily a corresponding result for subclasses of a given
hereditary graph class obtained by forbidding induced subgraphs.

\begin{theorem}\label{Tqufp}
Let $\cP$ be a hereditary graph class and define, for a family $\cF$ of
graphs,
\begin{equation}
\cP_\cF\=\set{G\in\cP:\tind(F,G)=0 \text{ for } F\in\cF}=\cP\cap\cU_\cF.
\end{equation}
Then
\begin{equation}
\overline{\cP_\cF}=
\bcp_\cF\=\set{\gG\in\bcp:\tind(F,\gG)=0 \text{ for } F\in\cF}.  
\end{equation}
\end{theorem}

\begin{proof}
  An immediate consequence of Theorems \ref{Tquf} and \ref{Tcap}.
\end{proof}

\begin{example}
The class $\UI$ of \emph{unit interval graphs}
equals also
the class of interval graphs that contain no induced subgraph isomorphic to
$K_{1,3}$
\cite[p.~111]{Brand}. 
Thus, if $\cI$ is the class of all interval graphs,
then $\UI=\cI_{\set{K_{1,3}}}$ and hence
$\overline{\UI}=\overline{\cI}_{\set{K_{1,3}}}$.
In other words,
if $\gG$ is a graph limit, then
$\gG\in\overline{\UI}$  if and only if $\gG\in\overline{\cI}$
and $\tind(K_{1,3},\gG)=0$.
\end{example}

We have here considered induced subgraphs. We obtain similar results if we
forbid general subgraphs. 
Let $\cU^*_\cF$ be the class of all graphs that 
do not contain any graph from $\cF$ as a subgraph,
\ie,
\begin{equation}\label{ufx}
\cU^*_\cF\=\set{G\in\cU:t(F,G)=0\text{ for } F\in\cF}.  
\end{equation}
This is evidently a hereditary class (and a monotone class, see \refE{mon}). 
In fact, this can be seen as a special
case of forbidding induced subgraphs, since $\cU^*_\cF=\cU_{\cfx}$, where
$\cfx$ is the family of all graphs $H$ that contain a spanning subgraph $F$
(\ie, a subgraph $F\subseteq H$ with $|F|=|H|$) with $F\in\cF$.

\begin{theorem}\label{Tqufx}
Let $\cU^*_\cF$ be given by \eqref{ufx}.
Then
\begin{equation}\label{qufx}
\overline{\cU^*_\cF}=\set{\gG\in\cuq:t(F,\gG)=0 \text{ for } F\in\cF}.  
\end{equation}
\end{theorem}

\begin{proof}
  By the argument in the proof of \refT{Tquf}, or by \refT{Tquf} applied to
  $\cF^*$. 
\end{proof}

There is also a version corresponding to \refT{Tqufp}.

\begin{example}\label{EK3b}
The class of \emph{triangle-free graphs} in \refE{EK3}
can also be defined as $\cU^*_{\set{K_3}}$.  
Thus  \refT{Tqufx}  applies and shows that the triangle free
  graph limits are the graph limits $\gG$ with $t(\gG,K_3)=0$. 
(This is in accordance with \refE{EK3} since
  $\tind(K_3,\gG)=t(K_3,\gG)$ for any graph limit $\gG$.)
\end{example}

\begin{example}
  The class  of \emph{bipartite graphs} equals $\cUU_{\set{C_3,C_5,\dots}}$.
Thus a graph limit $\gG$ is a limit of bipartite graphs if and only if
$t(C_k,\gG)=0$ for every odd $k\ge3$.
(We treat here bipartite graphs as a special case of simple graphs. 
Bipartite graphs, with an explicit bipartition, can also be treated
separately, with a corresponding but distinct limit theory, see \eg{}
\cite{SJ209}.) 
\end{example}

\section{Intersection graphs}\label{Sinter}

Consider the class $\pa$ of \aig{s} defined by a collection
$\cA$ of sets as in \refE{Einter}.
We define $W=W_\cA:\cA\times\cA\to\setoi$ by   
\begin{equation}\label{wa}
 W(A,B)=
 \begin{cases}
1 & \text{if } A\cap B\neq\emptyset,\\
0 & \text{if } A\cap B=\emptyset.
 \end{cases}
\end{equation}
Hence, a graph $G$ is an \aig{} if and only if there is a function $i\mapsto
A_i$ mapping $V(G)$ into $\cA$ such that the adjacency matrix of $G$ equals 
$(W(A_i,A_j))_{ij}$, except on the diagonal.

Equip $\cA$  with some suitable $\gs$-field making $W$
measurable and let 
$\mu$
be any probability measure on $\cA$. Then $W$ can be regarded as a
graphon defined on the probability space $(\cA,\mu)$, and defines thus, see
\refS{Sgraphons}, random graphs $G(n,W)$ and a graph limit $\gG$ such that
$G(n,W)\to\gG$ a.s.
By construction, the random graph $G(n,W)\in \pa$ for every $n$, and is thus
a random \aig; it follows that the graph limit $\gG\in\bpa$. 

This defines a graph limit $\gG\in\bpa$ for every probability measure $\mu$
on $\cA$. Note that we here use a fixed $W$ and vary $\mu$ to obtain
different graph limits, in contrast to the more common situation when $\mu$
is a fixed measure on some space $\sss$
(usually Lebesgue measure on \oi) and we vary $W$.
We therefore denote the graph limit obtained in this way from $\mu$ by
$\gG_\mu$. 

It is easily seen that every \aig{} $G$ can be obtained with positive
probability as $G(n,W)$ for a suitable $\mu$; if $G$ is defined by a family
$A_i\in\cA$, $i\in V(G)$, and $|V(G)|=n$,
take $\mu=\frac1n \sum_i \gd_{A_i}$, the
normalized sum of point masses at the points $A_i$.

The problem whether every \aigl{} can be represented as $\gG_\mu$ for some
probability measure $\mu$ on $\cA$
is more subtle.
We do not know any general results and give only some examples. 
Note that a class of intersection graphs typically
can be defined by several different families $\cA$, and it is conceivable
that the answer depends on the precise choice of the family $\cA$, and not just
on the resulting class of graphs. (Although we do not know any such
examples.)
We have not investigated this further, 
and we consider only some natural choices of $\cA$.

\begin{example}
  The class of \emph{interval graphs} can be defined as the \aigs{} where
  $\cA$ is the family \set{[a,b]:0\le a\le b\le 1}
of closed intervals of \oi. This family $\cA$ can be identified with a
closed subset of $\bbR^2$ (a triangle), and is then a compact metric space,
which we equip with the usual Borel $\gs$-field. 
It is shown in \cite{SJ254} that, with this choice of $\cA$,  every interval
graph limit is $\gG_\mu$ for some
(non-unique) probability measure $\mu$ on $\cA$.

Furthermore, \cite{SJ254} also gives similar results, 
with similar natural classes
$\cA$, for 
\emph{circular-arc graphs}, \emph{circle graphs} and \emph{permutation
  graphs}.
\end{example}

\begin{example}
However, it is also shown in \cite{SJ254} that the corresponding result
fails for the class of \emph{unit
  interval graphs}. This is the class of \aigs{} where
$\cA=\set{[x,x+1]:x\in\bbR}$, the family of closed unit intervals in
$\bbR$,
but not every unit interval graph limit can be represented as $\gG_\mu$ for
a probability measure $\mu$ on $\cA$.
(It is an open problem
whether there exists another family $\cA'$ defining the same class of graphs
such that every unit interval graph limit can be represented as $\gG_\mu$
for the family $\cA'$.)
\end{example}

\begin{example}\label{Ealla}
For a more trivial counterexample, 
let $\cA$ be the countable family of all finite 
subsets of $\bbN$; then every graph is an \aig.
(If $G$ is a graph, label the edges by integers and for each vertex $v$, let
$A_v$ be the set of the edges incident to $v$.)
However, a graph limit of the type $\gG_\mu$ is always random-free, see
\refS{Srf} below; since there are many graph limits that are not
random-free, not every graph limit can be represented as $\gG_\mu$.
\end{example}

\begin{example}\label{Echordal}
The class of \emph{chordal graphs}
defined in \refE{Echordal0} 
can also be defined as 
the class of intersection graphs of
subtrees in a tree  
\cite[Section 4.5]{Golumbic}. (In order to make this fit the formulation in
\refE{Einter}, we take a countably infinite universal tree $T$, containing
all finite 
trees as subtrees, for example constructed by taking disjoint copies of all
finite trees and joining them to a common root. We then let $\cA$ be the
family of all finite subtrees of $T$.)

We shall see in
\refE{Echordal2} that not every chordal graph limit is 
random-free, and thus not every chordal graph limit can be represented as
$\gG_\mu$
(for this or any other family $\cA$).

We do not know any natural representation of all
chordal graph limits, and leave
that as an open problem.
\end{example}

\section{Random-free graph limits and classes}\label{Srf}

A graphon $W$ is said to be \emph{random-free} if it is \oivalued\ a.e.
(In this case, the random graph $G(n,W)$ depends only on the random points
$X_1,X_2,\dots$, without further randomization, which is a reason for the
name.) If $W_1$ and $W_2$ are two graphons that represent the same graph
limit, and one of them is random-free, then both are, see \cite{SJ249} for a
detailed proof. Consequently, we can define a graph limit $\gG$ to be
random-free if some representing graphon is random-free, and in this case
every representing graphon is random-free.

Random-free graphons are studied in \cite[Section 10]{SJ249}, where it is
shown, for example, that a graph limit $\gG$ is random-free if and only if 
the random graph $G(n,\gG)$ has entropy $o(n^2)$ (thus quantifying a sense
in which there is less randomness than otherwise).
Another result, see \cite{Pikhurko} and \cite{SJ249}, 
is that $\gG$ is random-free if and only if
it is a
limit of a sequence of graphs in the stronger metric $\gd_1$.

\citet{LSz:regularity} define a graph property to be
\emph{random-free} if every graph limit $\gG\in\bcp$ is random-free.
They show some consequences of this. Moreover,
they show that
  a hereditary graph property $\cP$ is random-free if and only if
the following property holds:
\begin{property}\label{Prf}
There exists a bipartite graph $F$ with bipartition $(V_1,V_2)$ such that no
graph obtained from $F$ by adding edges within $V_1$ and $V_2$ is in $\cP$.
\end{property}

The representation theorems in \cite{SJ254} show that the class of
\emph{interval graphs} is random-free, together with the classes of 
\emph{circular-arc graphs} and \emph{circle graphs}. 
(Of course, then every subclass is also random-free, for
example the classes of \emph{unit interval graphs} and \emph{permutation graphs}
also studied in \cite{SJ254} and
the class of \emph{threshold graphs} for which random-free graphons were
found  in \cite{SJ238}.)

\begin{remark}
We do not know explicit examples of graphs $F$ satisfying \ref{Prf} for 
the random-free classes of graphs just mentioned, but we guess that small
examples exist. 
\end{remark}

However, not every class of intersection graphs is random-free. 

\begin{example}\label{Ealla2}
Let $\cA$ be the family of
finite subsets of some infinite set (for example $\bbN$). 
Then, as said in \refE{Ealla},
every graph is an $\cA$-intersection graph.  
Consequently, every graph limit is an $\cA$-intersection graph limit, and
the class of \aig{s} is not random-free.
\end{example}

\begin{example}\label{Echordal2}
The class of \emph{chordal graphs} 
can be defined in several ways, see \refE{Echordal0} and
\refE{Echordal}; it  is  an example both of a graph class defined by
forbidden induced subgraphs and a class of intersection graphs.

We show that this class does not satisfy \ref{Prf}; thus it is
\emph{not} random-free. 
The class of chordal graphs is thus a non-trivial class of intersection
graphs that is not 
random-free.

Let $F$ be a bipartite graph with bipartition $(V_1,V_2)$, and
let $F_1$ be the graph obtained 
by adding all edges inside $V_1$ to $F$.
We shall show that $F_1$ is chordal, which shows that \ref{Prf} does not hold.

Assume that a cycle $C$ is an induced subgraph of $F_1$. If $C$ has two
vertices in $V_1$, then these are adjacent in $F_1$ so they have to be
adjacent in $C$. Thus, $C$ has at most 3 vertices in $V_1$; if it has 3,
then $C$ has no other vertices, and if it has 2, they have to be adjacent in
$C$. Hence, there are at most 2 edges in $C$ than go between $V_1$ and
$V_2$. Since $F_1$, and thus $C$, has no edges inside $V_2$, it follows that
$C$ has at most one vertex in $V_2$. Consequently, $C$ has in every case at
most 3 vertices. 

We have shown that $F_1$ has no induced cycle of length greater than 3, \ie,
$F_1$ is chordal. Thus no graph $F$ as in \ref{Prf} exists, so by the result
of \citet{LSz:regularity} stated above, the class of chordal graphs is
\emph{not} random-free. 

In fact, we can easily construct a graphon that defines a chordal graph
limit but is not
random-free. Let $\sss=\setoi$ be a space with two points, say with measure
$\frac12$ each, and let $\wch$ be the graphon $\wch(x,y)\=(x+y)/2$ defined on
$\sss$. (Alternatively, one can take $\sss=\oi$ and
$\wch(x,y)\=(\floor{2x}+\floor{2y})/2$.) 
The random graph $G(n,\wch)$ is of the type of $F_1$ above (with $V_1$
the set of vertices $i$ with $X_i=1$); thus the argument above shows that
$G(n,\wch)$ is chordal. 
By \refT{Tgngg}, the graph limit $\ggx{\wch}$ generated by the graphon $\wch$ is
a chordal graph limit, and it is evidently not random-free.
(See also \refT{Tchordalrandom}.)
\end{example}

\begin{problem}
  Investigate for other classes of intersection graphs 
whether they are random-free or not.
\end{problem}

We turn to graph classes determined by forbidden subgraphs, noting that
the class of chordal graphs in
\refE{Echordal2} is one example of such a graph class that is not
random-free. 
Again, we have no general theorem and give
just one more example.

\begin{example}
The class $\CR$ of \emph{cographs}, see \refE{ECR}, is random-free.
This is shown in \citet{LSz:finitely} 
for regular graphons in $\overline{\CR}$;
we show the general case by verifying 
\ref{Prf},  using the result of \citet{LSz:regularity} stated above.
We take the graph $F$ as the bipartite graph with 12 vertices
$A,B,C,D,E,F,a,b,\allowbreak c,\allowbreak d,e,f$ and edge set 
$$
\set{Aa, Bb, Cc, Dd, Ee, Ff, Ac, Bd,  Ae, Bf, Ca, Cb, Db, Fa}.
$$
Suppose that there exists a graph $F_1$ obtained by adding edges within 
$V_1=\set{A,B,C,D,E,F}$ and $V_2=\set{a,b,c,d,e,f}$ to $F$ such that
$F_1\in\CR$, \ie, $F_1$ contains no induced $P_4$. 
If exactly one of $AB$ and $ab$ is an edge in $F_1$, then $aABb$ or $AabB$
is an induced $P_4$; hence either both are edges or neither is; let us write
this as $AB\iff ab$. Similarly, $CD\iff cd$, $EF\iff ef$, $AB\iff cd$,
$AB\iff ef$. Consequently, if $AB$ is an edge, then so are $ab$, $ef$ and $EF$,
and then $EFab$ is an induced $P_4$; conversely, if $AB$ is not an edge then
neither is $ab$ nor $CD$ and then $aCbD$ is an induced $P_4$.
Hence no such graph $F_1$ without induced $P_4$ exists, so \ref{Prf} holds.
\end{example}

\begin{problem}
  Investigate for other classes of graphs with forbidden (induced) subgraphs
whether they are random-free or not.
\end{problem}

\section{Disjoint cliques}\label{Sdc}

A \emph{clique} is a complete graph $K_n$, $n\ge1$. 
A \emph{\dc{} graph} is a graph that is a disjoint union of cliques.
(Note that this includes the possibility of isolated vertices.)
We let $\DC$ be the class of \dcg.

A graph is a \dcg{} if and only if the extended adjacency relation ($x\sim
y$ or $x=y$) is transitive, and thus an equivalence relation.
(We may thus also call these graphs \emph{transitive}.)
Consequently, a graph is a \dcg{} if and only if it has no induced subgraph
$P_3$, \ie, 
\begin{equation}\label{dcp3}
\DC=\cU_{\set{P_3}}
\=\set{G\in\cU:\tind(P_3,G)=0}.  
\end{equation}

Note also that $\DC$ is the class of \aig{s} where $\cA$ is any infinite
family of disjoint sets, for example $\cA=\set{\set i:i\in \bbN}$.

By \refT{Tquf},  \eqref{dcp3}
yields immediately the corresponding characterization of \dcgl{s}.
\begin{theorem}\label{TDC}
$\overline{\DC}=\cuq_{\set{P_3}}$. Hence,
  a graph limit $\gG$ is a \dcgl{} if and only if $\tind(P_3,\gG)=0$.
\nopf
\end{theorem}

We can obtain another, more explicit, description of the \dcgl{s} as
follows.
Note that a \dcg{} is determined by the partition of the vertex set into
connected components; an unlabelled \dcg{} is thus determined by the
sequence of component sizes.
In particular, the number of labelled \dcg{s} of order  $n$ is the Bell
number $B(n)$ and the number of unlabelled \dcg{s} of order $n$ is the number
$p(n)$ of partitions of $n$, see \eg{} \cite{FlajoletS}.
(Asymptotics are given in \cite[Example VIII.6 and Section VIII.6]{FlajoletS}.)

We denote for any graph $G$ the component sizes by $C_1(G)\ge
C_2(G)\ge\dots$, ordered in decreasing order and extended by $G_k(G)=0$ when
$k$ is larger than the number of components.

A \emph{mass-partition} is a sequence $\bs=(s_i)\iii$ of non-negative real
numbers such that
\begin{equation*}
s_1\ge s_2\ge\dots\ge0
\qquad \text{and} \qquad
\sumi s_i\le1.
\end{equation*}
(We think of breaking a unit mass into pieces of sizes $s_1,s_2,\dots$,
together with ``dust'', \ie, infinitesimal pieces, of mass $1-\sumi s_i\ge0$.)
The set $\MP$ of all mass-partitions is given the metric
$d(\bs,\bs')\=\max_i|s_i-s_i'|$ and is then a compact metric space; the
topology equals the topology of pointwise convergence (\ie, convergence of
each $s_i$ separately); for these and other properties, see \cite[Section
  2.1]{Bertoin}. 

Note that each graph $G$ defines a mass-partition $(C_i(G)/|G|)\iii$.

We define, for a mass-partition $\bs$,
\begin{equation}
S_k(\bs)\=\sumi s_i^{k},
\qquad k\ge1.
\end{equation}
\begin{lemma}\label{LMP}
  \begin{thmenumerate}
  \item \label{lmp1}
  If $\bs$ and $\bs'$ are mass-partitions, then $\bs=\bs'$ if and only if
$S_k(\bs)=S_k(\bs')$ for every $k\ge2$.
\item \label{lmp2}
  If $\bs\nn$ is a sequence of mass-partitions, then $\bs\nn\to\bs$ in $\MP$
if and only if
$S_k(\bs\nn)\to S_k(\bs)$ for every $k\ge2$.
  \end{thmenumerate}
\end{lemma}

\begin{proof}
  \pfitemref{lmp1}
We have $s_1=\sup_i s_i=\lim_{k\to\infty} S_k(\bs)^{1/k}$. Hence
$S_k(\bs)=S_k(\bs')$ for all large $k$ implies $s_1=s_1'$, and it follows by
induction that $s_i=s_i'$ for all $i$. The converse is obvious.

  \pfitemref{lmp2}
Since $\bs\in\MP$ implies $s_i\le 1/i$, it follows by dominated convergence
that if $\bs\nn\to\bs$, then $S_k(\bs\nn)\to S_k(\bs)$ for each $k\ge2$.

For the converse we note that this means that the
map $\Phi:\bs\mapsto (S_k(\bs)_{k=2}^\infty$ is a
continuous map of $\MP$ into $\oi^\infty$. Moreover, by  \ref{lmp1}, this
map is one-to-one. Since $\MP$ is compact, it follows that $\Phi$ is a
homeomorphism. 
\end{proof}

\begin{remark}
In \refL{LMP}\ref{lmp1}, we then also have $S_1(\bs)=S_1(\bs')$, but in
\ref{lmp2}, it does not follow that $S_1(\bs)\to S_1(\bs')$; for an example,
let $s_i\nn=1/n$ for $i\le n$; then $\bs\nn\to0$ but $S_1(\bs\nn)=1$ for
each $n$.  
\end{remark}

Given a mass-partition $\bs$ we define a graphon $\ws$ by taking disjoint
subsets $(A_i)\iii$ of a \ps{} $(\sss,\mu)$ such that $\mu(A_i)=s_i$ and
defining 
$\ws\=\sumi \etta_{A_i\times A_i}$.
It is easily seen from the definition \eqref{tfw} (or \eqref{tfwp}) that if
$F$ is connected and $|F|=k\ge2$, then
\begin{equation}\label{tws}
  t(F,\ws)=S_k(\bs)\=\sumi s_i^{k}.
\end{equation}
Since $t(K_1,W)=1$ for any graphon $W$, it follows by \eqref{fcomp} that 
the graphons $\ws$ defined in this way all are equivalent and thus define a
unique graph limit $\ggs$. (This is an example of a direct sum of
(infinitely many) graph limits, see \cite{SJ213}.)

\begin{remark}
One canonical choice of $\ws$ is to take $\sss \=\bbN\cup\set\infty$ with
$\mu\set{i}=s_i$, $i<\infty$, and thus $\mu\set\infty=1-\sumi s_i$,
letting $A_i=\set i$. Then $\ws(x,y)=\ett{x=y<\infty}$.
(An related construction is to take $\ws(x,y)\=\ett{x=y}$ on a \ps{} with
point masses $s_1,s_2,\dots$, and the rest of the mass in a continuous part.)

Another canonical choice is to take $(\sss,\mu)=(\oi,\gl)$ and let
$A_i=[\sum_1^{i-1} s_j,\sum_1^{i} s_j)$.  
\end{remark}

\begin{theorem}\label{TDC2}
 The set $\DCoo$ of \dcgl{} is the set $\set{\ggs:\bs\in\MP}$, and the
 mapping $\bs\mapsto\ggs$ is a homeomorphism of $\MP$ onto $\DCoo$.

If\/ $G_n$ is a sequence of disjoint clique
graphs with $|G_n|\to\infty$ and 
$\gG$ is a graph limit, 
then
$G_n\to\gG$ if and only if\/
$\gG=\ggs$ for some 
$\bs\in\MP$, and
$(C_i(G_n)/|G_n|)_i\to\bs$ in $\MP$.
\end{theorem}

\begin{proof}
\refL{LMP}, \eqref{tws} and \eqref{fcomp} show that $\bs\mapsto t(F,\ws)$ is
continuous for every graph $F$, which is the same as saying that
$\bs\to\ggs$ is a continuous map $\MP\to\cuoo$. Since $\MP$ is compact and
the map is one-to-one by \eqref{tws} and \refL{LMP}, it is a homeomorphism
onto a closed subset of $\cuoo$, which we temporarily denote by $K$.

Let $G_n$ be a sequence of disjoint clique graphs with $G_n\to\infty$
and let
$\bcn$ be the mass-partition 
$(C_i(G_n)/|G_n|)_i$ defined by $G_n$.
It is easy to see that
if $F$ is connected and $|F|=k\ge2$, then, using \eqref{tws},
\begin{equation}
t(F,G_n)=S_k(\bcn)+O(1/|G_n|)  = S_k(\bcn)+o(1)
=t(F,\ggMP_{\bcn})+o(1).
\end{equation}
Hence, using \eqref{fcomp}, for any graph limit $\gG$, 
\begin{equation}\label{ros}
  G_n\to\gG
\iff
\ggMP_{\bcn}\to\gG.
\end{equation}
Since $\ggMP_{\bcn}\in K$ and $K$ is a closed subset of $\cuoo$, it follows
that if $G_n\to\gG$, then $\gG\in K$, so $\gG=\ggMP_{\bs}$ for some $\bs\in\MP$.
Furthermore, since we have shown that $\bs\mapsto\ggMP_{\bs}$ is an
homeomorphism,
\eqref{ros} also implies that $G_n\to\ggs$ if and only if $\bcn\to\bs$.
It follows also that $\xoo{\DC}=K$.
\end{proof}

\begin{example}\label{E01}
  The extreme cases of \dcg{s} are the complete graph $K_n$ and the empty
  graph $E_n=\overline{K_n}$; 
they converge as \ntoo{} to the graph limits $\ggi$ and 
$\ggo$
defined by
the constant graphons 1 and 0, respectively.
Note that $\ggi$ and $\ggo$ are the graph limits $\ggs$  defined by the
mass-partitions $(1,0,0,\dots)$ and $(0,0,\dots,)$.
(The families \set{K_n} and \set{E_n} are themselves examples of hereditary
classes, although quite trivial.)
\end{example}

We can easily obtain limit results for uniformly random clique graphs. We
give both an unlabelled and a labelled version.

\begin{theorem}\label{Trcg}
  \begin{thmenumerate}
  \item 
 If $G_n$ is a uniformly random unlabelled \dcg{} of order $n$, then
$G_n\pto \ggo$ as \ntoo.
\item 
If $H_n$ is a uniformly
random labelled \dcg{} of order $n$, then
$H_n\pto \ggo$ as \ntoo.
  \end{thmenumerate}
\end{theorem}

\begin{proof}
(i): In this case,
$(C_i(G_n))_i$ is a uniformly random partition of $n$, and it is well-known that
$C_1(G_n)$ is of the order $\sqrt n \log n$ (see \citet{ErdosL} and
\citet{Fristedt} for more precise results). Hence $C_1(G_n)/n\pto0$ and 
thus $(C_i(G_n)/n)_i\pto(0,0,\dots)$ in $\MP$.
Consequently, $G_n\pto \ggo$ by \refT{TDC2}.

(ii):
Similarly, 
$C_1(H_n)$ is the size of the largest part in a uniformly random set
partition of $[n]$.
It follows from the asymptotics of the Bell numbers that
$C_1(H_n)/n\pto0$, see \citet{Sachkov} for a much more precise result.
Thus, similarly to (i),  $H_n\pto \ggo$ by \refT{TDC2}.
\end{proof}

\begin{remark}\label{Rug}
For any graphs $G_n$ with $|G_n|\to\infty$,
 $G_n\to\ggo$ if and only if the number of edges $e(G_n)=o(n^2)$.
Hence, these results for random \dcg{s} are equivalent to
$e(G_n)/n^2\pto0$ and $e(H_n)/n^2\pto0$.
\end{remark}

\begin{example}\label{EDC1}
  Let $\dci$ be the subclass of $\DC$ consisting of disjoint clique
graphs with at most one non-trivial clique, \ie, the graphs that are the
disjoint union of a $K_m$ and an $E_{n-m}$, with $1\le m\le n$.
Thus, if $G\in\cD$, then $C_2(G)\le1$. It follows easily from \refT{TDC2}
that the set $\xoo{\dci}$ is the subset of $\DCoo$ consisting of the graph
limits $\ggd_t\=\ggMP_{\bs_t}$, $t\in\oi$,
where $\bs_t$ is the mass-partition $(t,0,\dots)$; note that $\ggMP_{\bs_t}$
is represented by the graphon $\etta_{[0,t]\times[0,t]}$ on $(\oi,\gl)$.
(Thus, $\ggd_0=[0]$ and $\ggd_1=[1]$, \cf{} \refE{E01}.)

An unlabelled graph in $\dci$ is determined by the numbers $n$ and $m$ above.
Hence, if we let $G_n$ be a uniformly random unlabelled graph of order $n$ in
$\dci$, then $m$ is uniformly distributed over \set{1,\dots,n}. It follows
that $G_n\dto \ggd_T$, where the random graph limit $\ggd_T$ has $T\in\oi$
random and uniformly distributed. This is thus an example of a hereditary
class where a uniformly random graph $G_n$ converges in distribution to some
random (and non-deterministic) graph limit.

On the other hand, it is easily seen that a uniformly random labelled graph
$G_n$ in $\dci$ converges in probability (and thus in distribution)
to the non-random $\ggd_{1/2}$.
In fact, since the decomposition $K_m\cup E_{n-m}$ is unique when $m\ge2$,
it is easily seen that the random graph $G(n,\ggd_{1/2})$ is almost uniformly
distributed, in the sense that the total variation distance
$\dtv(G_n,G(n,\ggd_{1/2}))\to0$ as \ntoo;  thus $G_n\pto \ggd_{1/2}$
follows from \eqref{gnwlim}.
\end{example}

\section{Line graphs}\label{Slg}

The \emph{line graph} $L(G)$
of a graph $G$ has as vertices the edges of $G$, with
$e$ and $f$ adjacent in $L(G)$ if they have a common endpoint in $G$.
Let $\LG$ be the class of line graphs.

There are several other, equivalent, characterizations, see \cite[Theorem
  7.1.8]{Brand}. In particular, there is a set $\cflg$ of 9 graphs such that
a graph is a line graph if and only if it does not have any induced subgraph
in $\cflg$, \ie,
$\LG=\cU_{\cflg}$. (See also \citet{Soltes} and \citet{LaiSoltes}.)
\refT{Tquf} yields immediately the corresponding characterization of \dcgl{s}.
\begin{theorem}\label{TLG}
$\overline{\LG}=\cuq_{\cflg}$. Hence,
  a graph limit $\gG$ is a \lgl{} if and only if $\tind(F,\gG)=0$ for
  $F\in\cflg$. 
\nopf
\end{theorem}

The line graph of a star is a complete graph, and therefore every 
\dcg{} is a line graph (\viz{} the line graph of a disjoint union of
stars). There are many other line graphs, but we shall see that the line
graph limits are the same as the limits of the subclass of \dcg{s}.

\begin{lemma}\label{LLG}
  If $G$ is a line graph of order $n$, then $G$ has a subgraph $H$ that is
  a \dcg{} 
  with $V(H)=V(G)$ 
and $|E(G)\setminus E(H)|\le 4 n^{5/3}=o(n^2)$.
\end{lemma}

\begin{proof}
The line graph  $G$ is a union of edge-disjoint cliques $\cC_i$, with each
vertex in at most  two $\cC_i$. 
(If $G=L(G')$, then $\cC_i$ is the set of edges in $G'$ incident to a vertex
$i$ in $G'$.)
Note that thus $\sum_i|\cC_i|\le2n$.

Let $H_1\=V(G)\cup\bigcup\set{\cC_i:|\cC_i|>n\qqqb}$, 
where $V(G)$ is seen as an empty graph;
\ie, $H$ equals $G$ with all edges in
cliques $\cC_i$ with $|\cC_i|\le n\qqqb$ removed.
The number of removed edges is
\begin{equation*}
  \begin{split}
  |E(G)\setminus E(H_1)|
& =\sum_{|\cC_i|\le n\qqqb} \binom{|\cC_i|}{2}
\le\sum_{|\cC_i|\le n\qqqb}|\cC_i|^{2}
\le n\qqqb\sum_{|\cC_i|\le n\qqqb}|\cC_i|
\\&
\le 2n\cdot n\qqqb = 2 n^{5/3}.	
  \end{split}
\end{equation*}

Since $\sum_i|\cC_i|\le 2n$, there are at most $2n/n\qqqb=2n\qqq$ remaining
cliques.
Since two cliques have at most one common vertex (they are edge-disjoint),
the number of vertices that belong to 2 cliques is at most
$\binom{2n\qqq}{2}\le 2n\qqqb$. Delete also all edges incident to any these
vertices. This 
leaves a graph $H\subseteq H_1$ that is a union of disjoint cliques, and
\begin{equation*}
  |E(H_1)\setminus E(H)| \le 2 n\qqqb\cdot n = 2 n^{5/3}.
\end{equation*}
Hence $|E(G)\setminus E(H)| \le 4 n^{5/3}$.
\end{proof}

\begin{theorem}\label{TLG=}
  $\xoo{\LG}=\xoo{\DC}$,
\ie, a graph limit is a \lgl{} if and only if it is a \dcgl. 
Hence, the set $\xoo\LG$ of \lgl{s} is
  the set \set{\ggs:\bs\in\MP}.
\end{theorem}

\begin{proof}
  Let $G_n$ be a sequence of line graphs with $G_n\to\gG$.
By \refL{LLG}, we may select disjoint clique graphs $H_n\subseteq G_n$
such that $V(H_n)=V(G_n)$ and $|E(G_n)\setminus V(H_n)|=o(|G_n|^2)$.
It follows easily that for any graph $F$, $t(F;G_n)=t(F,H_n)+o(1)$, and thus
$t(F,H_n)=t(F,\gG)+o(1)$; hence $H_n\to\gG$. This shows that any line graph
limit $\gG$ is a \dcgl. The converse is obvious.
\end{proof}

\begin{corollary}\label{CLG=}
  If $\gG$ is a graph limit, then $\tind(F,\gG)=0$ for every $F\in\cF_L$ if and
  only if $\tind(P_3,\gG)=0$.
\end{corollary}

\begin{proof}
  By Theorems \ref{TDC}, \ref{TLG} and \ref{TLG=}.
\end{proof}

\begin{remark}
Note that \refC{CLG=} holds for graph limits but not for graphs.
(There are graphs $G$ with
 $\tind(F,G)=0$ for every $F\in\cF_L$ but $\tind(P_3,G)>0$, for example
$G=P_3$.)

Furthermore, we see that every line graph limit $\gG$ satisfies
$\tind(P_3,\gG)=0$, although line graphs may contain $P_3$. 
Note that if $G_n\to\gG$,  then $\tind(P_3,G_n)\to \tind(P_3,\gG)=0$, 
and it follows easily (using compactness and subsequences; we omit the
details) that if $G_n$ is any sequence of line graphs with $|G_n|\to\infty$,
then $\tind(P_3,G_n)\to0$. \refL{LLG} implies that
$\tind(P_3,G_n)=O(|G_n|\qqqw)$; we do not know whether this rate
is the best possible.
\end{remark}

We conjecture that a uniformly random line graph (labelled or unlabelled)
converges in probability to $\ggo$, just as random disjoint clique
graphs do by \refT{Trcg}, but we have not
investigated this further.

\section{Ramsey's theorem}\label{SRamsey}

Ramsey's theorem says that for every $r\ge1$, every sufficiently large graph
contains either $K_r$ or its complement $E_r$ as an induced
subgraph. (See further \cite{Ramsey}.)
In the notation of \refS{Sforbidden}, the theorem says that the
graph class $\cU_{\set{K_r,E_r}}$ is finite, which by Theorems \ref{T0} and
\ref{Tquf} are equivalent to the following:

\begin{theorem}[Ramsey's theorem for graph limits]\label{Tramsey1}
  For every $r\ge1$,
 \begin{equation*}
\bbcu_{\set{K_r,E_r}}\=\set{\gG\in\bbcu:\tind(K_r,\gG)=\tind(E_r,\gG)=0}
=\emptyset.	
  \end{equation*}
\end{theorem}

\begin{proof}
We have just given a proof from the classical Ramsey theorem.
We find it instructive to also give a direct proof using graph limits; this
thus yields a graph limit proof of Ramsey's theorem.
(This is a simple analogue of more advanced results such as the hypergraph
removal lemma that can be proved using (hyper)graph limits, see \eg{} 
\cite{ElekSz} and \cite{Tao}.)

By \eqref{tindfw}, the statement can be formulated as: there is no graphon
$W$ such that 
\begin{equation}\label{ramsey}
  \prod_{1\le i<j\le r} W(x_i,x_j)
=
  \prod_{1\le i<j\le r} (1-W(x_i,x_j))
=0
\end{equation}
for \aex{} $x_1,\dots,x_r\in\oi$.
To see that \eqref{ramsey} yields a contradiction, the problem is the
``a.e.'',
which we handle by \refL{LD} below, which shows that we may modify $W$ on a
null set such that \eqref{ramsey} holds for all $x_1,\dots,x_r$ such that
$(x_i,x_j)\in E$ for some set $E$ of full measure and containing the
diagonal \set{(x,x)}. We may then take $x_1=\dots=x_r=x$ for any fixed
$x\in\oi$, and obtain $W(x,x)^{\binom r2} = (1-W(x,x))^{\binom r2}=0$,
and thus $W(x,x)=0$ and $W(x,x)=1$, a contradiction.
\end{proof}

The required lemma is a version of \cite[Lemma 5.3]{SJ234} for
several polynomials simultaneously. 
Since the version in \cite{SJ234} is stated for a single polynomial, 
we give a detailed statement.
A \emph{multiaffine} polynomial is a polynomial in several
variables $\set{x_\nu}_{\nu\in\cI}$, for some
(finite) index set $\cI$, such that each variable has degree at
most 1; it can thus be written as a sum
$\sum_{\cJ\subseteq\cI} a_{\cJ}\prod_{\nu\in \cJ} x_\nu$.

\begin{lemma}\label{LD}
Suppose that\/ $W:\oi^2\to\oi$ is a graphon.
Then there is a version\/ $W'$ of\/ $W$, \ie{} a 
graphon $W'$ such that $W'=W$ \aex{},
and a symmetric set  $E\subseteq\oi^2$
such that $\gl(\oi^2\setminus E)=0$ and $E\supseteq\set{(x,x):x\in\oi}$, and, 
moreover, if $\Phi\bigpar{(w_{ij})_{i<j}}$ is a multiaffine polynomial in the
$\binom m2$ variables $w_{ij}$, $1\le  i<j\le m$, for some $m\ge2$,
such that $\Phi\bigpar{(W(x_i,x_j))_{i<j}}=\gam$ 
for \aex{} $\xxm\in\oi$ and some $\gam\in\bbR$,
then $\Phi\bigpar{(W'(x_i,x_j))_{i<j}}=\gam$ 
for all  $x_1,\dots,x_m$ such that $(x_i,x_j)\in E$ for
  every pair $(i,j)$ with $1\le i<j\le m$.
\end{lemma}

\begin{proof}
  This is proved in \cite[Appendix]{SJ234} for a single polynomial $\Phi$.
However, an inspection of the proof shows that $W'$ and $E$ are constructed
independently of $\Phi$, so the same choices work for any such $\Phi$.
\end{proof}

Furthermore, it is easy to see that Ramsey's theorem is equivalent to the
following:

\begin{theorem}[Ramsey's theorem in disguise]\label{Tramsey2}
  If $\cP$ is an infinite hereditary graph class, then either $\cP$ contains
  every complete graph $K_n$, $n\ge1$, 
or $\cP$ contains every empty graph $E_n$, $n\ge1$.
\end{theorem}
\begin{proof}
  Since $\cP$ is hereditary and contains arbitrarily large graphs, Ramsey's
  theorem implies that for every $r$, $K_r\in\cP$ or $E_r\in\cP$. Thus at
  least one of $K_r\in\cP$ or $E_r\in\cP$ holds for arbitrarily large $r$,
  and thus for all $r$ since $\cP$ is hereditary.
\end{proof}

As a corollary we immediately obtain the following simple result.
(Compare the weaker but more general \refT{T0}.)

\begin{theorem}\label{Tramsey3}
  If $\cP$ is an infinite hereditary graph class, then $\cpoo$ contains either
 $\ggo$  or $\ggi$ (or both). 
\end{theorem}
\begin{proof}
  By \refT{Tramsey2}, since $K_n\to\ggi$ and $E_n\to\ggo$ as \ntoo.
\end{proof}

Conversely, \refT{Tramsey3} implies \refT{Tramsey2} by \refT{Tsofie}, so
\refT{Tramsey3} may also be regarded as a graph limit version of Ramsey's
theorem. 

\refT{Tramsey3} exhibits a minimum content of $\cpoo$ for
a hereditary class $\cP$. It is best possible; 
the hereditary classes \set{K_n} and \set{E_n} 
show that $\cpoo=\set{[1]}$ and $\cpoo=\set{[0]}$ both are possible.
(Cf.~\refE{E01}.)
Furthermore,
\refT{Tclaw} gives an example with
$\cpoo=\set{\ggo,\ggi}$ and the class $\cQ$ defined after it is another example
with $\cqoo=\set{\ggo}$.

\section{Split graphs}

A \emph{split graph} is a graph whose vertex set can be partitioned as
$V_0\cup V_1$ such that the subgraph induced by $V_0$ is empty and the
subgraph induced by $V_1$ is complete, see \cite{chordalsplit}. 
Let $\SP$ denote the class of split
graphs; this is evidently a hereditary class.

\begin{theorem}\label{Tsplit}
  A graph limit is a split graph limit if and only if it can be represented
  by a graphon $W$ such that, for some $a\in\oi$,
$W=0$ on $[0,a]\times[0,a]$ and 
$W=1$ on $(a,1]\times(a,1]$.
\end{theorem}

\begin{proof}
  If $W$ is a graphon of this type, then $G(n,W)$ is \as{} a split graph.
  (Take $V_0=\set{i:X_i\in[0,a]}$ in the construction in \refS{Sgraphons}.)
Thus the graph limit $\ggw$ generated by $W$ belongs to $\hSP$ by \refT{Tgngg}.

Conversely, let $G_n$ be split graphs with $|G_n|\to\infty$ and $G_n\to W$
for some graphon $W$. Let $V(G_n)$
have the partition $V(G_n)_0\cup V(G_n)_1$ as above, and let
$a_n\=|V(G_n)_0|/|V(G_n)|$.  
Order the vertices of $G_n$ with $V(G_n)_0$ first, and
let $W_{G_n}(x,y)\=A_{G_n}(\ceil{nx},\ceil{ny})$ be the corresponding
graphon, where $A_{G_n}$ is the adjacency matrix of $G_n$ (see \eg{}
\cite{BCLSVi} or \cite{SJ249} for this standard construction of a graphon
corresponding to a graph).
Then $\dcut(W_{G_n},W)\to0$ as \ntoo, where $\dcut$ is the cut metric, see
\eg{} \cite{BCLSVi} or \cite{SJ249}. Furthermore, 
$\iint_{[0,a_n]^2} W_{G_n}=0$ and 
$\iint_{(a_n,1]^2}(1-W_{G_n})\to0$, and it follows from the definition of
the cut metric that there exist sets $A_n\subseteq\oi$ with $\gl(A_n)=a_n$
such that $\iint_{A_n^2} W\to0$ and  $\iint_{(\oi\setminus A_n)^2}
(1-W)\to0$.

Consider a subsequence such that the indicator functions $\etta_{A_n}$
converge in the weak${}^*$ topology in $L^\infty(\oi)$ (as the dual of
$L^1(\oi)$ to some function $g\in L^\infty(\oi)$; this is possible by the
compactness of the unit ball in the weak${}^*$ topology.
It is easily seen that
then
$\iint g(x)g(y)W(x,y)=0$ and $\iint(1-g(x))(1-g(y))(1-W(x,y))=0$. Hence, if
$A\=\set{x:g(x)>0}$, then $W=0$ \aex{} on $A\times A$ and $W=1$ \aex{} on
  $(\oi\setminus A)^2$, and $W$ is equivalent to a graphon of the desired
type by a measure-preserving rearrangement of \oi.
\end{proof}

We can also describe the limit of a uniformly random split graph.
 Recall the graphon $\wch$ in \refE{Echordal2}.

\begin{theorem}\label{Tsplitrandom}
Let $G_n$ be a random (labelled or unlabelled) split graph of order $n$.
Then, $G_n\pto\wch$ as \ntoo.
\end{theorem}

\begin{proof}
This follows by \refT{Tsplit} and \cite[Theorem 1.6]{HJSz}, since 
among the graphons in \refT{Tsplit}, there is (up to \aex{} equivalence) 
a unique graphon that maximizes the entropy defined in \cite{HJSz}, 
\viz{}  $\wch$ (regarded as a graphon on \oi).

It is also possible to give a direct proof; we give a sketch. 
Consider first the labelled case.

The partition $V_0\cup V_1$ of a split graph is not always unique, but
different such partitions can differ in at most two points. It is also
easily seen that most labelled split graphs have a unique partition of this
type, in the sense that the fraction of them among all labelled split graphs
of order $n$ tends to 1 as \ntoo, and that $|V_0|$ is concentrated around
$n/2$. Hence, we can construct a random graph $G_n'$, with distribution
approximating $G_n$ in the sense that the total variation distance
$\dtv(G_n,G_n')\to0$, by first selecting $V_0$ with a suitable distribution
among all subsets of size in $(n/2-n^{3/4},n/2+n^{3/4})$, say, and then
choosing the edges between $V_0$ and $V_1\=[n]\setminus V_0$ at random,
independently and with probability $1/2$ each. It follows that $G_n'\pto
\wch$, and thus $G_n\pto\wch$ as \ntoo; we omit the details.

The unlabelled case follows from the labelled, since most split graphs have
a trivial automorphism group; again we omit the details.  
\end{proof}

It is easily seen that every split graph is chordal; \cf{} \refE{Echordal2}.
It is easy to produce examples of chordal graphs that do not split;
for example, any disjoint clique graph. \refT{TDC2} yields plenty of chordal
graph limits that are not split graph limits, since they are not represented
by any graphon as in \refT{Tsplit}.
Nevertheless, \citet{chordalsplit} have shown that most (labelled) chordal
graphs 
split, and thus \refT{Tsplitrandom} yields the following.

\begin{theorem}\label{Tchordalrandom}
Let $G_n$ be a random (labelled or unlabelled) chordal graph of order $n$.
Then, $G_n\pto\wch$ as \ntoo.
\end{theorem}

\begin{proof}
  The labelled case follows immediately from \refT{Tsplitrandom} and 
\cite{chordalsplit}. The unlabelled case follow from this because, as said
above,
most split graphs have
a trivial automorphism group.
\end{proof}

\section{Claw-free and coclaw-free graphs}\label{Sclaw}

The \emph{claw} is the star $K_{1,3}$ with 4 vertices, and the \emph{coclaw}
is its complememt $\overline{K_{1,3}}$, \ie, the disjoint union of a
triangle $K_3$ and an isolated vertex.

The \emph{claw-free} graphs are the graphs without an induced claw, \ie,
$\cU_{\set{K_{1,3}}}$. 
(See \eg{} \cite{Brand} for this class of graphs.)
By \refT{Tquf}, a graph limit $\gG$ is a claw-free
graph limit (\ie, belongs to the closure $\overline{\cU}_{\set{K_{1,3}}}$)
if and only if $t(K_{1,3},\gG)=0$.
Similarly, a graph limit $\gG$ is coclaw-free 
if and only if $t(\overline{K_{1,3}},\gG)=0$.
Note that a graphon $W$ is coclaw-free if and only if
$1-W$ is claw-free.

We do not know any simple characterization of claw-free graphons.
(And thus not of coclaw-free graphons.)
However, it is easy to characterize graph limits that are 
{both} claw-free and
coclaw-free; they turn out to be trivial. 

\begin{theorem}\label{Tclaw}
A graphon that is both claw-free and coclaw-free has to be either 
$0$ or $1$ a.e. 
Thus, the only  graph limits that are both claw-free and coclaw-free are
$\ggo$ and $\ggi$.
\end{theorem}

\begin{proof}
  Let $W:\oi^2\to\oi$ be a claw-free graphon. By \eqref{tindfw}, 
  \begin{multline}\label{www}
W(x_1,x_2)W(x_1,x_3)W(x_1,x_4)
\\
\times\bigpar{1-W(x_2,x_3)}	\bigpar{1-W(x_2,x_4)}	
\bigpar{1-W(x_3,x_4)}	=0
  \end{multline}
for \aex{} $x_1,x_2,x_3,x_4\in\oi$.
By \refL{LD} we may 
(by modifying $W$ on a null set)
assume that there exists a symmetric set
$E\in\oi^2$ with Lebesgue measure 1, and containing the diagonal
$\set{(x,x)}$, such that \eqref{www} holds for all $x_1,x_2,x_3,x_4$ such
that all pairs $(x_i,x_j)\in E$.

Taking $x_1=x_2=x_3=x_4=x$, we see that $W(x,x)=0$ or 1 for every $x$.
Let $A_0\=\set{x:W(x,x)=0}$ and $A_1\=\set{x:W(x,x)=1}$.
Moreover, if $x\in A_0$, we see  by taking $x_1=y$ and $x_2=x_3=x_4=x$
that $W(x,y)=0$ when $(x,y)\in E$. Thus $W(x,y)=0$ for \aex{} $(x,y)\in
A_0\times\oi$. 

If furthermore $W$ also is coclaw-free, the same argument applies to $1-W$
(using the same modification of $W$).
Thus also $W(x,y)=1$ for \aex{} $(x,y)\in
A_1\times\oi$. 

It follows that $0=W=1$ \aex{} on $A_0\times A_1$, and thus
$A_0\times A_1$ is a null set, so either $\gl(A_0)=1$ and
$W=0$ \aex, or $\gl(A_1)=1$ and $W=1$ a.e.
\end{proof}

This result for graph limits translates to a
corresponding result for graphs. 
In fact, \refT{Tclaw} implies that a graph that is both claw-free and
coclaw-free has to be 
either almost empty or almost complete. More precisely, for any sequence
$G_n$ of such graphs, with $|G_n|=n$ for simplicity, the number of edges
$e(G_n)$ satisfies 
$e(G_n)=o(n^2)$ or $\binom n2-e(G_n)=o(n^2)$, see \refR{Rug}.
We give a simple direct
proof of this, with an explicit (and much sharper) error bound; we find it
interesting to compare the two very different proofs of 
the same result.

  Let $\cQ$ be the class of graphs whose components are 
cycles of length $\ge4$ and
paths. 
$\cQ$ can also be described as the class of graphs $G$ with maximum degree
$\gD(G)\le 2$ and no component $K_3$; note that $\cQ$ is a hereditary class.
Furthermore, if $G\in \cQ$ and $|G|=n$, then $2e(G)\le n\gD(G)\le 2n$, and thus
$e(G)\le n$. 

Note that the estimate $e(G)\le|G|$ just given for graphs $G\in\cQ$ implies
that any sequence of graphs $G_n\in\cQ$ with $|G_n|\to\infty$ converges to
$\ggo$; thus the only $\cQ$ graph limit is $\ggo$, \ie,
$\xoo{\cQ}=\set{\ggo}$.  

Let $R(k,l)$ denote the Ramsey numbers, see \cite{Ramsey}.
\begin{theorem}\label{Tclaw2}
If $G$ is a graph with $|G|=n\ge R(8,8)$, then $G$ is claw-free and
coclaw-free if and only if $G\in\cQ$ or $\overline G\in\cQ$.
In particular, then either $e(G)\le n$ or $e(G)\ge\binom n2-n$.
\end{theorem}

\begin{remark}
  It is known that $R(8,8)\le\binom{14}7=3432$, see \cite[Section
	4.3]{Ramsey}. This lower bound for the validity of the conclusion is
  presumably too high, but note that the result is not true for $3\le n\le 6$;
  counterexamples  are provided by a triangle with 0--3 additional
  vertices connected to one vertex each in the triangle.
\end{remark}

\begin{remark}
  Note that \refT{Tclaw} is an immediate consequence of \refT{Tclaw2}:
By \refT{Tquf} (or \refT{Tcap}) 
a graph limit $\gG$ that is claw-free and
coclaw-free  is a limit of a sequence of graphs that are
claw-free and coclaw-free, and by \refT{Tclaw2}, this implies that 
$\gG$ is the limit of a sequence $G_n$ of graphs such that either 
$G_n\in\cQ$ or $\overline{G_n}\in\cQ$. Selecting a subsequence we thus have
either $G_n\to\ggo$ or
$G_n\to\ggi$. 
\end{remark}

\begin{proof}
It is clear that if $G\in\cQ$, then $G$ is claw-free and coclaw-free, and
thus the same holds if $\overline G\in\cQ$.

For the converse we note that  since $n\ge R(8,8)$, either $G$ or its
complement $\overline G$ contains an empty induced  subgraph $E_8$. 
We assume that $E_8$ is an induced subgraph of $G$ and show that if $G$
further is claw-free and coclaw-free, then $G\in\cQ$.
This follows from the two claims below.

\smallskip\emph{Claim 1.}
$G$ contains no induced $K_3$.

In fact, if $G$ contains a $K_3$ and an $E_8$, then they may have at most
one common vertex and thus there exist two vertex disjoint subgraphs $A\cong
K_3$ and $B\cong E_7$ in $G$.
Every vertex in $B$ has to send at least one edge to $A$; otherwise it would
form a $\overline{K_{1,3}}$ with $A$. Hence there are at least 7 edges
between $A$ and $B$.
On the other hand, if some vertex in $A$ is connected to 3 (or more)
vertices in $B$, it forms a $K_{1,3}$ together with them; hence each vertex
in $A$ is connected to at most two vertices in $B$ and the number of edges
betwenn $A$ and $B$ is at most 6.
This contradiction proves the claim.

\smallskip\emph{Claim 2.}
The maximum degree $\gD(G)\le2$.

In fact, for any vertex $v\in G$, Claim 1 shows
that the neighbours of $v$ form an independent set.
If $G$ has a vertex $v$ with degree 3 or more, then $v$ and any three of its
neighbours thus form a $K_{1,3}$.
\end{proof}

A graph is both claw-free and coclaw-free if and only if its complement is.
Hence \refT{Tclaw} and symmetry implies the following, giving another 
example where a uniformly random graph in a hereditary class has a limit in
distribution that is random and not a single, deterministic graph limit.

\begin{theorem}
  Let $G_n$ be a uniformly random (labelled or unlabelled) claw-free and
  coclaw-free graph. Then $G_n\dto\gG$ \asntoo, where $\gG$ is the random
  graph limit that equals $\ggo$ and $\ggi$ with probability $1/2$ each.
\nopf
\end{theorem}

\newcommand\AAP{\emph{Adv. Appl. Probab.} }
\newcommand\JAP{\emph{J. Appl. Probab.} }
\newcommand\JAMS{\emph{J. \AMS} }
\newcommand\MAMS{\emph{Memoirs \AMS} }
\newcommand\PAMS{\emph{Proc. \AMS} }
\newcommand\TAMS{\emph{Trans. \AMS} }
\newcommand\AnnMS{\emph{Ann. Math. Statist.} }
\newcommand\AnnPr{\emph{Ann. Probab.} }
\newcommand\CPC{\emph{Combin. Probab. Comput.} }
\newcommand\JMAA{\emph{J. Math. Anal. Appl.} }
\newcommand\RSA{\emph{Random Struct. Alg.} }
\newcommand\ZW{\emph{Z. Wahrsch. Verw. Gebiete} }
\newcommand\DMTCS{\jour{Discr. Math. Theor. Comput. Sci.} }

\newcommand\AMS{Amer. Math. Soc.}
\newcommand\Springer{Springer-Verlag}
\newcommand\Wiley{Wiley}

\newcommand\vol{\textbf}
\newcommand\jour{\emph}
\newcommand\book{\emph}
\newcommand\inbook{\emph}
\def\no#1#2,{\unskip#2, no. #1,} 
\newcommand\toappear{\unskip, to appear}

\newcommand\webcite[1]{\url{#1}}
\newcommand\webcitesvante{\webcite{http://www.math.uu.se/~svante/papers/}}
\newcommand\arxiv[1]{\webcite{http://arxiv.org/#1}}

\def\nobibitem#1\par{}

\end{document}